  \documentclass[multi]{cambridge7A}
  \usepackage[numbers]{natbib}

  \usepackage{rotating}
  \usepackage{floatpag}
  \rotfloatpagestyle{empty}

 \usepackage{amsmath,amsfonts,mathrsfs,amssymb}
  \usepackage{amsthm}
  \usepackage{graphicx}

  \usepackage{multind}
  \makeindex{authors}
  \makeindex{subject}

  \theoremstyle{plain}
  \newtheorem{theorem}{Theorem}[chapter]
  \newtheorem{lemma}[theorem]{Lemma}
  \newtheorem*{corollary}{Corollary}
  \newtheorem{conjecture}[theorem]{Conjecture}
  \newtheorem{proposition}[theorem]{Proposition}

  \theoremstyle{definition}
  \newtheorem{definition}[theorem]{Definition}
  \newtheorem{example}[theorem]{Example}

  \theoremstyle{remark}
  
  \newtheorem*{note}{Note}

\newcommand{\MHM}{\operatorname{MHM}}

\newcommand{\pt}{\mathit{pt}}
\newcommand{\Dmod}{\mathscr{D}}
\newcommand{\Mmod}{\mathcal{M}}
\newcommand{\Nmod}{\mathcal{N}}

\newcommand{\derR}{\mathbf{R}}
\newcommand{\derL}{\mathbf{L}}

\newcommand{\Ltensor}{\overset{\derL}{\tensor}}
\newcommand{\decal}[1]{\lbrack #1 \rbrack}
\newcommand{\shH}{\mathcal{H}}

\newcommand{\abs}[1]{\lvert #1 \rvert}

\newcommand{\tensor}{\otimes}

\newcommand{\shHom}{\mathcal{H}\hspace{-1pt}\mathit{om}}

\newcommand{\ZZ}{\mathbb{Z}}
\newcommand{\QQ}{\mathbb{Q}}

\newcommand{\CC}{\mathbb{C}}
\newcommand{\HH}{\mathbb{H}}

\newcommand{\menge}[2]{\bigl\{ \thinspace #1 \thinspace\thinspace \big\vert%
\thinspace\thinspace #2 \thinspace \bigr\}}
\newcommand{\Menge}[2]{\Bigl\{ \thinspace #1 \thinspace\thinspace \Big\vert%
\thinspace\thinspace #2 \thinspace \Bigr\}}

\DeclareMathOperator{\im}{im}

\DeclareMathOperator{\Spec}{Spec}

\DeclareMathOperator{\Supp}{Supp}
\DeclareMathOperator{\codim}{codim}

\DeclareMathOperator{\rat}{rat}

\DeclareMathOperator{\DR}{DR}

\DeclareMathOperator{\Hom}{Hom}

\DeclareMathOperator{\GL}{GL}

\DeclareMathOperator{\Pic}{Pic}

\newcommand{\define}[1]{\emph{#1}}
\newcommand{\shf}[1]{\mathscr{#1}}

\newcommand{\omegaX}{\omega_X}

\newcommand{\into}{\hookrightarrow}

\newcommand{\jl}{j_{\ast}}

\newcommand{\jlsl}{j_{!\ast}}
\newcommand{\ju}{j^{\ast}}
\newcommand{\fu}{f^{\ast}}
\newcommand{\fl}{f_{\ast}}

\newcommand{\iu}{i^{\ast}}
\newcommand{\ius}{i^{!}}
\newcommand{\il}{i_{\ast}}

\newcommand{\pl}{p_{\ast}}

\newcommand{\fp}{f_{+}}

\newcommand{\theoremref}[1]{Theorem~\ref{#1}}
\newcommand{\lemmaref}[1]{Lemma~\ref{#1}}
\newcommand{\definitionref}[1]{Definition~\ref{#1}}
\newcommand{\propositionref}[1]{Proposition~\ref{#1}}
\newcommand{\conjectureref}[1]{Conjecture~\ref{#1}}

\newcommand{\OA}{\shf{O}_A}

\newcommand{\OAsh}{\shf{O}_{\Ash}}
\newcommand{\Ash}{A^{\natural}}
\newcommand{\Bsh}{B^{\natural}}

\newcommand{\fsh}{f^{\natural}}
\newcommand{\fsi}{f^{+}}
\renewcommand{\mp}{m_+}
\newcommand{\Db}{\mathrm{D}^{\mathit{b}}}
\newcommand{\Dbcoh}{\Db_{\mathit{coh}}}
\newcommand{\Dtcoh}[1]{\mathrm{D}_{\mathit{coh}}^{#1}}
\newcommand{\pshH}{ {^p} \shH}
\newcommand{\ppshH}{ {^{p_+}} \! \shH}

\newcommand{\Dbh}{\Db_{\mathit{h}}}
\newcommand{\Dbc}{\Db_{\mathit{c}}}

\newcommand{\Dtppc}[1]{ {^{p_+}} \! \mathrm{D}_c^{#1}}
\renewcommand{\Hom}{\operatorname{Hom}}

\DeclareMathOperator{\Aut}{Aut}
\renewcommand{\shHom}{\mathcal{H}\mathit{om}}
\newcommand{\DA}{\mathbf{D}_A}

\DeclareMathOperator{\Char}{Char}
\newcommand{\CCrho}{\CC_{\rho}}
\newcommand{\Psh}{P^{\natural}}
\newcommand{\nablash}{\nabla^{\natural}}
\DeclareMathOperator{\FM}{FM}
\newcommand{\CCst}{\CC^{\ast}}
\renewcommand{\shH}{\mathcal{H}}
\renewcommand{\HH}{\mathbf{H}}
\newcommand{\opp}{\mathit{opp}}
\newcommand{\QQb}{\bar{\QQ}}
\newcommand{\shL}{\mathcal{L}}
\newcommand{\shLC}{\shL_{\CLambda}}
\newcommand{\shLQ}{\shL_{\QLambda}}
\newcommand{\shLk}{\shL_{\kLambda}}
\newcommand{\QLambda}{\QQ \lbrack \Lambda \rbrack}
\newcommand{\CLambda}{\CC \lbrack \Lambda \rbrack}
\newcommand{\kLambda}{k \lbrack \Lambda \rbrack}
\DeclareMathOperator{\Gal}{Gal}
\newcommand{\GalCQ}{\Gal(\CC/\QQ)}
\DeclareMathOperator{\ord}{ord}
\newcommand{\Lp}{L \lbrack p \rbrack}
\newcommand{\Eom}{E_{\omega}}

\begin{document}




  \alphafootnotes
  \author[Christian Schnell]{Christian Schnell}

  \chapterauthor{Christian Schnell}



  \chapter[Torsion points on cohomology support loci]%
		{Torsion points on cohomology support loci:\\%
		from D-modules to Simpson's theorem}
	\chaptermark{Torsion points on cohomology support loci}

  \contributor{Christian Schnell
    \affiliation{Department of Mathematics,
		Stony Brook University,
      Stony Brook, NY 11794, USA}}

\begin{abstract}
We study cohomology support loci of regular holonomic $\Dmod$-modules on complex
abelian varieties, and obtain conditions under which each irreducible component of
such a locus contains a torsion point. One case is that both the $\Dmod$-module and the
corresponding perverse sheaf are defined over a number field; another case is that
the $\Dmod$-module underlies a graded-polarizable mixed Hodge module with a
$\ZZ$-structure.  As a consequence, we obtain a new proof for Simpson's result that
Green-Lazarsfeld sets are translates of subtori by torsion points.
\end{abstract}

\section{Overview of the paper}

\subsection{Introduction}

Let $X$ be a projective complex manifold. In their two influential
papers about the generic vanishing theorem \cite{GL1,GL2}, Green and Lazarsfeld
showed that the so-called \define{cohomology support loci}
\[
	\Sigma_m^{p,q}(X) = \menge{L \in \Pic^0(X)}{%
		\dim H^q \bigl( X, \Omega_X^p \tensor L \bigr) \geq m},
\]
are finite unions of translates of subtori of $\Pic^0(X)$. Beauville and Ca\-ta\-ne\-se
\cite{Beauville} conjectured that the translates are always by \emph{torsion
points}, and this was proved by Simpson \cite{Simpson} with the help
of the Gelfond-Schneider theorem from transcendental number theory.
There is also a proof using positive characteristic methods by
Pink and Roessler \cite{PinkRoessler}. 

Over the past ten years, the results of Green and Lazarsfeld have been reinterpreted
and generalized several times \cite{Hacon,PopaSchnell,Schnell}, and we now understand
that they are consequences of a general theory of holonomic $\Dmod$-modules on
abelian varieties.  The purpose of this paper is to investigate under what conditions
the result about torsion points on cohomology support loci remains true in that
setting. One application is a new proof for the conjecture by Beauville and Catanese
that does not use transcendental number theory or reduction to positive
characteristic. 

\begin{note}
In a recent preprint \cite{Wang}, Wang extends \theoremref{thm:main} to 
polarizable Hodge modules on compact complex tori; as a corollary, he proves the
conjecture of Beauville and Catanese for arbitrary compact K\"ahler manifolds.
\end{note}

\subsection{Cohomology support loci for D-modules}

Let $A$ be a complex abelian variety, and let $\Mmod$ be a regular holonomic
$\Dmod_A$-module; recall that a $\Dmod$-module is called \define{holonomic} if its
characteristic variety is a union of Lagrangian subvarieties of the cotangent
bundle. Denoting by $\Ash$ the moduli space of line bundles with integrable
connection on $A$, we define the \define{cohomology support loci} of $\Mmod$
as
\[
	S_m^k(A, \Mmod) = \Menge{(L, \nabla) \in \Ash}{\dim \HH^k \Bigl( A, 
		\DR_A \bigl( \Mmod \tensor_{\OA} (L, \nabla) \bigr) \Bigr) \geq m}
\]
for $k, m \in \ZZ$. It was shown in \cite[Theorem~2.2]{Schnell} that 
$S_m^k(A, \Mmod)$ is always a finite union of linear subvarieties of $\Ash$, in the
following sense.

\begin{definition}
A \define{linear subvariety} of $\Ash$ is any subset of the form 
\[
	(L, \nabla) \tensor \im \bigl( \fsh \colon \Bsh \to \Ash \bigr),
\]
for $f \colon A \to B$ a homomorphism of abelian varieties, and $(L, \nabla)$ a point
of $\Ash$. An \define{arithmetic subvariety} is a linear subvariety that contains a
torsion point.
\end{definition}

Moreover, the analogue of Simpson's theorem is true for semisimple regular holonomic
$\Dmod$-modules of geometric origin: for such $\Mmod$, every irreducible component of
$S_m^k(A, \Mmod)$ contains a torsion point. We shall generalize this result in two
directions:
\begin{enumerate}
\item Suppose that $\Mmod$ is regular holonomic, and that both $\Mmod$ and the
corresponding perverse sheaf $\DR_A(\Mmod)$ are defined over a number field. We shall
prove that the cohomology support loci of $\Mmod$ are finite unions of arithmetic
subvarieties; this is also predicted by Simpson's standard conjecture.
\item Suppose that $\Mmod$ underlies a graded-polarizable mixed Hodge module with
$\ZZ$-structure; for example, $\Mmod$ could be the intermediate extension of a
polarizable variation of Hodge structure with coefficients in $\ZZ$. We shall prove
that the cohomology support loci of $\Mmod$ are finite unions of arithmetic
subvarieties.
\end{enumerate}

\subsection{Simpson's standard conjecture}

In his article \cite{Simpson-RH}, Simpson proposed several conjectures about
regular holonomic systems of differential equations whose monodromy representation
is defined over a number field. The principal one is the so-called ``standard
conjecture''; restated in the language of regular holonomic $\Dmod$-modules and
perverse sheaves, it takes the following form.
\begin{conjecture} \label{conj:Simpson}
Let $\Mmod$ be a regular holonomic $\Dmod$-module on a smooth projective 
variety $X$, both defined over $\QQb$. If $\DR_X(\Mmod)$ is the complexification of a
perverse sheaf with coefficients in
$\QQb$, then $\Mmod$ is of geometric origin.
\end{conjecture}
He points out that, ``there is certainly no more reason to believe it is true than to
believe the Hodge conjecture, and whether or not it is true, it is evidently
impossible to prove with any methods which are now under consideration. However, it
is an appropriate motivation for some easier particular examples, and it leads to
some conjectures which might in some cases be more tractable'' \cite[p.~372]{Simpson-RH}.

In the particular example of abelian varieties, \conjectureref{conj:Simpson} predicts
that when $\Mmod$ is a regular holonomic $\Dmod$-module with the properties described
in the conjecture, then the cohomology support loci of $\Mmod$ should be finite
unions of arithmetic subvarieties. Our first result -- actually a simple consequence
of \cite{Schnell} and an old theorem by Simpson \cite{Simpson} -- is that this
prediction is correct.

\begin{theorem}
Let $A$ be an abelian variety defined over $\QQb$, and let $\Mmod$ be a regular
holonomic $\Dmod_A$-module. If $\Mmod$ is defined over $\QQb$, and if 
$\DR_A(\Mmod)$ is the complexification of a perverse sheaf with coefficients in
$\QQb$, then all cohomology support loci $S_m^k(A, \Mmod)$ are finite unions of
arithmetic subvarieties of $\Ash$.
\end{theorem}

\begin{proof}
Let $\Char(A) = \Hom \bigl( \pi_1(A), \CCst \bigr)$ be the space of rank one
characters of the fundamental group; for a character $\rho \in \Char(A)$, we denote
by $\CCrho$ the corresponding local system of rank one. We define the
cohomology support loci of a constructible complex of $\CC$-vector spaces $K \in
\Dbc(\CC_A)$ to be the sets
\[
	S_m^k(A, K) = \Menge{\rho \in \Char(A)}{\dim \HH^k \bigl( A, K \tensor_{\CC}
		\CCrho \bigr) \geq m}.
\]
The correspondence between local systems and vector bundles with integrable
connection gives a biholomorphic mapping $\Phi \colon \Ash \to \Char(A)$; it takes a
point $(L, \nabla)$ to the local system of flat sections of $\nabla$. According to
\cite[Lemma~14.1]{Schnell}, the cohomology support loci satisfy
\[
	\Phi \bigl( S_m^k(A, \Mmod) \bigr) = S_m^k \bigl( A, \DR_A(\Mmod) \bigr).
\]
Note that $\Char(A)$ is an affine variety defined over $\QQ$; in our situation,
$\Ash$ is moreover a quasi-projective variety defined over $\QQb$, because the same is
true for $A$. The assumptions on the $\Dmod$-module $\Mmod$ imply that $S_m^k(A, \Mmod)
\subseteq \Ash$ is defined over $\QQb$, and that $S_m^k \bigl( A, \DR_A(\Mmod) \bigr)
\subseteq \Char(A)$ is defined over $\QQb$. We can now use
\cite[Theorem~3.3]{Simpson} to conclude that both must be finite unions of arithmetic
subvarieties.
\end{proof}

\subsection{Mixed Hodge modules with $\ZZ$-structure}

We now consider a much larger class of regular holonomic $\Dmod_A$-modules, namely
those that come from mixed Hodge modules with $\ZZ$-structure. This class includes,
for example, intermediate extensions of polarizable variations of Hodge structure
defined over $\ZZ$; the exact definition can be found in \definitionref{def:ZZ}
below. 

We denote by $\MHM(A)$ the category of graded-polarizable mixed Hodge modules
on the abelian variety $A$, and by $\Db \MHM(A)$ its bounded derived category
\cite[\S4]{Saito-MHM}; because $A$ is projective, every mixed Hodge module is
automatically algebraic. Let
\[
	\rat \colon \Db \MHM(A) \to \Dbc(\QQ_A)
\]
be the functor that takes a complex of mixed Hodge modules to the underlying
complex of constructible sheaves of $\QQ$-vector spaces; then a $\ZZ$-structure on $M
\in \Db \MHM(A)$ is a constructible complex $E \in \Dbc(\ZZ_A)$ with the property
that $\rat M \simeq \QQ \tensor_{\ZZ} E$. 

To simplify the notation, we shall define the \define{cohomology support loci} of a
complex of mixed Hodge modules $M \in \Db \MHM(A)$ as
\[
	S_m^k(A, M) = \menge{\rho \in \Char(A)}{%
		\dim H^k \bigl( A, \rat M \tensor_{\QQ} \CCrho \bigr) \geq m},
\]
where $k \in \ZZ$ and $m \geq 1$. Our second result is the following structure
theorem for these sets.

\begin{theorem} \label{thm:main}
If a complex of mixed Hodge modules $M \in \Db \MHM(A)$ admits a $\ZZ$-structure,
then all of its cohomology support loci $S_m^k(A, M)$ are complete unions of arithmetic
subvarieties of $\Char(A)$.
\end{theorem}

\begin{definition} \label{def:complete}
A collection of arithmetic subvarieties of $\Char(A)$ is called \define{complete}
if it is a finite union of subsets of the form
\[
	\menge{\rho^k}{\gcd(k,n) = 1} \cdot 
		\im \bigl( \Char(f) \colon \Char(B) \to \Char(A) \bigr),
\]
where $\rho \in \Char(A)$ is a point of finite order $n$, and $f \colon A
\to B$ is a surjective morphism of abelian varieties with connected fibers.
\end{definition}

The proof of \theoremref{thm:main} occupies the remainder of the paper; it is by
induction on the dimension of the abelian variety. Since we already know that the
cohomology support loci are finite unions of linear subvarieties, the issue is to
prove that every irreducible component contains a torsion point. Four important ingredients
are the Fourier-Mukai transform for $\Dmod_A$-modules \cite{Laumon,Rothstein};
results about Fourier-Mukai transforms of holonomic $\Dmod_A$-modules \cite{Schnell};
the theory of perverse sheaves with integer coefficients \cite{BBD}; and of course
Saito's theory of mixed Hodge modules \cite{Saito-MHM}. Roughly speaking, they make
it possible to deduce the assertion about torsion points from the following
elementary special case: if $V$ is a graded-polarizable variation of mixed Hodge
structure on $A$ with coefficients in $\ZZ$, and if $\CCrho$ is a direct factor of $V
\tensor_{\ZZ} \CC$ for some $\rho \in \Char(A)$, then $\rho$ must be a torsion
character.  The completeness of the set of components follows from the fact that
$S_m^k(A, M)$ is defined over $\QQ$, hence stable under the natural $\GalCQ$-action;
note that the $\GalCQ$-orbit of a character $\rho$ of order $n$ consists exactly of
the characters $\rho^k$ with $\gcd(k,n) = 1$.

\subsection{The conjecture of Beauville and Catanese}

Now let $X$ be a projective complex manifold. As a consequence of
\theoremref{thm:main}, we obtain a purely analytic proof for the conjecture of
Beauville and Catanese. 

\begin{theorem} \label{thm:GL}
Each $\Sigma_m^{p,q}(X)$ is a finite union of subsets of the form $L \tensor T$,
where $L \in \Pic^0(X)$ is a point of finite order, and $T \subseteq \Pic^0(X)$ is a
subtorus.
\end{theorem}

\begin{proof}
Inside the group $\Char(X)$ of rank one characters of the fundamental group of $X$, let
$\Char^0(X)$ denote the connected component of the trivial character. If $f \colon X
\to A$ is the Albanese morphism (for some choice of base point on $X$), then
$\Char(f) \colon \Char(A) \to \Char^0(X)$ is an isomorphism. As above, we denote the
local system corresponding to a character $\rho \in \Char(X)$ by the symbol $\CCrho$.
Define the auxiliary sets
\[
	\Sigma_m^k(X) = \menge{\rho \in \Char^0(X)}{\dim \HH^k \bigl( X, \CCrho \bigr) \geq m};
\]
by the same argument as in \cite[Theorem~3]{Arapura}, it suffices to prove that each
$\Sigma_m^k(X)$ is a finite union of arithmetic subvarieties of $\Char^0(X)$. But this
follows easily from \theoremref{thm:main}. To see why, consider the complex of mixed
Hodge modules $M = \fl \QQ_X^H \decal{\dim X} \in \Db \MHM(A)$. The underlying
constructible complex is $\rat M = \derR \fl \QQ \decal{\dim X}$, and so
\[
	\Sigma_m^{k + \dim X}(X) = \Char(f) \bigl( S_m^k(A, M) \bigr).
\]
Because $\derR \fl \ZZ \decal{\dim X}$ is a $\ZZ$-structure on $M$, the assertion is
an immediate consequence of \theoremref{thm:main}.
\end{proof}

For some time, I thought that each $\Sigma_m^{p,q}(X)$ might perhaps also be
complete in the sense of \definitionref{def:complete}, meaning a finite union of
subsets of the form
\[
	\menge{L^{\tensor k}}{\gcd(k,n) = 1} \tensor T,
\]
where $n$ is the order of $L$. Unfortunately, this is not the case.

\begin{example}
Here is an example of a surface $X$ where certain cohomology support loci are not
complete.  Let $A$ be an elliptic curve. Choose a nontrivial character $\rho \in
\Char(A)$ of order three, let $L = \CCrho \tensor_{\CC} \OA$, and let $B \to A$ be the \'etale
cover of degree three that trivializes $\rho$. The Galois group of this cover is $G =
\ZZ / 3 \ZZ$, and if we view $G$ as a quotient of $\pi_1(A, 0)$, then the three characters
of $G$ correspond exactly to $1, \rho, \rho^2$. Finally, let $\omega$ be a primitive
third root of unity, and let $\Eom$ be the elliptic curve with an automorphism of
order three. Now $G$ acts diagonally on the product $\Eom \times B$, and the quotient
is an isotrivial family of elliptic curves $f \colon X \to A$.  Let us consider the
variation of Hodge structure on the first cohomology groups of the fibers. Setting $H
= H^1(\Eom, \ZZ)$, the corresponding representation of the fundamental group of $A$
factors as
\[
	\pi_1(A, 0) \to G \to \Aut(H),
\]
and is induced by the $G$-action on $\Eom$. This representation is the direct sum of
the two characters $\rho$ and $\rho^2$, because $G$ acts as multiplication by
$\omega$ and $\omega^2$ on $H^{1,0}(\Eom)$ and $H^{0,1}(\Eom)$, respectively.
For the same reason, $\fl \omegaX \simeq L$ and $R^1 \fl \omegaX \simeq \OA$.
Since $\fu \colon \Pic^0(A) \to \Pic^0(X)$ is an isomorphism, the projection formula
gives
\[
	\Sigma_1^{2,0}(X) = \bigl\{ L^{-1} \bigr\} \quad \text{and} \quad
		\Sigma_1^{2,1}(X) = \bigl\{ L^{-1}, \OA \bigr\}.
\]
We conclude that not all cohomology support loci of $X$ are complete.
\end{example}

\begin{note}
Although he does not state his result in quite this form, Pareschi
\cite[Scholium~4.3]{Pareschi} shows that the set of positive-dimensional irreducible
components of 
\[
	V^0(\omegaX) = \Sigma_1^{\dim X,0}(X)
\] 
is complete, provided that $X$ has maximal Albanese dimension. 
\end{note}

\subsection{Acknowledgements}

This work was supported by the World Premier International Research Center Initiative
(WPI Initiative), MEXT, Japan, and by NSF grant DMS-1331641. I thank Nero Budur,
Fran\c{c}ois Charles, and Mihnea Popa for several useful conversations about the
topic of this paper, and Botong Wang for suggesting to work with arbitrary complexes
of mixed Hodge modules. I am indebted to an anonymous referee for suggestions
about the exposition. I also take this opportunity to thank Rob Lazarsfeld for the
great influence that his work with Mark Green and Lawrence Ein has had on my
mathematical interests.

\section{Preparation for the proof}

\subsection{Variations of Hodge structure}

In what follows, $A$ will always denote a complex abelian variety, and $g = \dim A$
its dimension. To prove \theoremref{thm:main}, we have to show that certain complex
numbers are roots of unity; we shall do this with the help of Kronecker's theorem,
which says that if all conjugates of an algebraic integer have absolute value
$1$, then it is a root of unity. To motivate what follows, let us consider the
simplest instance of \theoremref{thm:main}, namely a polarizable variation of Hodge
structure with coefficients in $\ZZ$. 

\begin{lemma} \label{lem:VHS}
If a local system with coefficients in $\ZZ$ underlies a polarizable variation of
Hodge structure on $A$, then it is a direct sum of torsion points of $\Char(A)$.
\end{lemma}

\begin{proof}
The associated monodromy representation $\mu \colon \pi_1(A) \to \GL_n(\ZZ)$,
tensored by $\CC$, is semisimple \cite[\S4.2]{Deligne-II}; the existence of a
polarization implies that it is isomorphic to a direct sum of unitary characters of
$\pi_1(A)$. Since $\mu$ is defined over $\ZZ$, the collection of these characters is
preserved by the action of $\GalCQ$. This means that the values of each character, as
well as all their conjugates, are algebraic integers of absolute value $1$; by
Kronecker's theorem, they must be roots of unity. It follows that $\mu$ is a direct
sum of torsion characters.
\end{proof}

\begin{corollary} \label{cor:VHS}
Let $V$ be a local system of $\CC$-vector spaces on $A$. If $V$ underlies a
polarizable variation of Hodge structure with coefficients in $\ZZ$, 
all cohomology support loci of $V$ are finite unions of arithmetic subvarieties.
\end{corollary}

\begin{proof}
By Lemma~\ref{lem:VHS}, we have $V \simeq \CC_{\rho_1} \oplus \dotsb \oplus
\CC_{\rho_n}$ for torsion points $\rho_1, \dotsc, \rho_n \in \Char(A)$. All
cohomology support loci of $V$ are then obviously contained in the set
\[
	\bigl\{ \rho_1^{-1}, \dotsc, \rho_n^{-1} \bigr\},
\]
and are therefore trivially finite unions of arithmetic subvarieties.
\end{proof}

\subsection{Mixed Hodge modules with $\ZZ$-structure}

We shall say that a mixed Hodge module has a $\ZZ$-structure if the underlying
perverse sheaf, considered as a constructible complex with coefficients in $\QQ$, can
be obtained by extension of scalars from a constructible complex with coefficients in
$\ZZ$. A typical example is the intermediate extension of a variation of Hodge
structure with coefficients in $\ZZ$. To be precise, we make the following definition.

\begin{definition} \label{def:ZZ}
A \define{$\ZZ$-structure} on a complex of mixed Hodge modules 
\[
	M \in \Db \MHM(A)
\]
is a constructible complex $E \in \Dbc(\ZZ_A)$ such that $\rat M \simeq \QQ
\tensor_{\ZZ} E$.
\end{definition}

The standard operations on complexes of mixed Hodge modules clearly respect
$\ZZ$-structures. For instance, suppose that $M \in \Db \MHM(A)$ has a
$\ZZ$-structure, and that $f \colon A \to B$ is a homomorphism of abelian varieties; then
$\fl M \in \Db \MHM(B)$ again has a $\ZZ$-structure. The proof is straightforward:
\[
	\rat \bigl( \fl M \bigr) = \derR \fl(\rat M)
		\simeq \derR \fl \bigl( \QQ \tensor_{\ZZ} E \bigr)
		\simeq \QQ \tensor_{\ZZ} \derR \fl E
\]
 
By \cite[Section~3.3]{BBD} and \cite{Juteau}, there are two natural perverse
t-structures on the category $\Dbc(\ZZ_A)$; after tensoring by $\QQ$, both become
equal to the usual perverse t-structure on $\Dbc(\QQ_A)$. We shall use the one
corresponding to the perversity $p_+$; concretely, it is defined as follows:
\[
	E \in \Dtppc{\leq 0}(\ZZ_A) \Longleftrightarrow
	\Bigg\{ \, \parbox[c]{.5\textwidth}{%
		for any stratum $S$, the local system \\
      $\shH^m \iu_S E$ is zero if $m > - \dim S + 1$, \\
		and $\QQ \tensor_{\ZZ} \shH^{-\dim S + 1} \iu_S E = 0$}
\]
\[
	E \in \Dtppc{\geq 0}(\ZZ_A) \Longleftrightarrow
	\Bigg\{ \, \parbox[c]{.5\textwidth}{%
		for any stratum $S$, the local system \\
      $\shH^m \ius_S E$ is zero if $m < - \dim S$, \\
		and $\shH^{-\dim S} \ius_S E$ is torsion-free}
\]
We can use the resulting formalism of perverse sheaves with integer coefficients to
show that $\ZZ$-structures are also preserved under taking cohomology.

\begin{lemma} \label{lem:shH}
If $M \in \Db \MHM(A)$ admits a $\ZZ$-structure, then each cohomology module
$\shH^k(M) \in \MHM(A)$ also admits a $\ZZ$-structure.
\end{lemma}

\begin{proof}
Let $\ppshH^k(E)$ denote the $p_+$-perverse cohomology sheaf in degree $k$ of the
constructible complex $E \in \Dbc(\ZZ_A)$. With this notation, we have
\[
	\rat \shH^k(M) = \pshH^k(\rat M) \simeq \pshH^k \bigl( \QQ \tensor_{\ZZ} E \bigr) 
		\simeq \QQ \tensor_{\ZZ} \ppshH^k(E),
\]
which gives the desired $\ZZ$-structure on $\shH^k(M)$.
\end{proof}

There is also a notion of intermediate extension for local systems with integer
coefficients. If $i \colon X \into A$ is a subvariety of $A$, and $j \colon U \into X$ is
a Zariski-open subset of the smooth locus of $X$, then for any local system $V$ on
$U$ with coefficients in $\ZZ$, one has a canonically defined $p_+$-perverse sheaf
\[
	\il \bigl( \jlsl V \decal{\dim X} \bigr) \in 
		\Dtppc{\leq 0}(\ZZ_A) \cap \Dtppc{\geq 0}(\ZZ_A).
\]
After tensoring by $\QQ$, it becomes isomorphic to the usual intermediate extension
of the local system $\QQ \tensor_{\ZZ} V$. This has the following immediate
consequence.

\begin{lemma} \label{lem:IC}
Let $M$ be a polarizable Hodge module. Suppose that $M$ is the intermediate extension
of $\QQ \tensor_{\ZZ} V$, where $V$ is a polarizable variation of Hodge structure
with coefficients in $\ZZ$. Then $M$ admits a $\ZZ$-structure.
\end{lemma}

\begin{proof}
In fact, $E = \il \bigl( \jlsl V \decal{\dim X} \bigr)$ gives a $\ZZ$-structure on $M$.
\end{proof}

We conclude our discussion of $\ZZ$-structures by improving \lemmaref{lem:VHS}.

\begin{lemma} \label{lem:factor}
Let $M$ be a mixed Hodge module with $\ZZ$-structure. Let $\rho \in \Char(A)$ be a
character with the property that, for all $g \in \GalCQ$, the local system $\CC_{g
\rho} \decal{\dim A}$ is a subobject of $\CC \tensor_{\QQ} \rat M$. Then $\rho$ is
a torsion point of $\Char(A)$.
\end{lemma}

\begin{proof}
Let $j \colon U \into A$ be the maximal open subset with the property that $\ju M = V
\decal{\dim A}$ for a graded-polarizable variation of mixed Hodge structure $V$.
Consequently, $\ju \CCrho$ embeds into the complex variation of mixed Hodge
structure $\CC \tensor_{\QQ} V$. Since the variation is graded-polarizable, and since
$\jl \colon \pi_1(U) \to \pi_1(A)$ is surjective, it follows that $\rho$ must be
unitary \cite[\S1.12]{Deligne}. On the other hand, we have $\rat M \simeq \QQ \tensor_{\ZZ}
E$ for a constructible complex $E$ with coefficients in $\ZZ$. Then $H^{-\dim A} \ju
E$ is a local system with coefficients in $\ZZ$, and $\ju \CCrho$ embeds into
its complexification. The values of the character $\rho$ are therefore algebraic
integers of absolute value $1$. We get the same conclusion for all their conjugates,
by applying the argument above to the characters $g \rho$, for $g \in \GalCQ$. Now
Kronecker's theorem shows that $\rho$ takes values in the roots of unity, and is
therefore a torsion point of $\Char(A)$.  
\end{proof}

\subsection{The Galois action on the space of characters}

In this section, we study the natural action of $\GalCQ$ on the space of
characters, and observe that the cohomology support loci of a regular holonomic
$\Dmod$-module with $\QQ$-structure are stable under this action.

The space of characters $\Char(A)$ is an affine algebraic variety, and its coordinate
ring is easy to describe. We have $A = V / \Lambda$, where $V$ is a complex
vector space of dimension $g$, and $\Lambda \subseteq V$ is a lattice of rank $2g$;
note that $\Lambda$ is canonically isomorphic to the fundamental group
$\pi_1(A, 0)$. For a field $k$, we denote by
\[
	\kLambda = \bigoplus_{\lambda \in \Lambda} k e_{\lambda}
\]
the group ring of $\Lambda$ with coefficients in $k$; the product is determined by
$e_{\lambda} e_{\mu} = e_{\lambda + \mu}$. As a complex algebraic variety,
$\Char(A)$ is isomorphic to $\Spec \CLambda$; in particular, $\Char(A)$ can already be
defined over $\Spec \QQ$, and therefore carries in a natural way an action of the
Galois group $\GalCQ$.

\begin{proposition} \label{prop:Galois-Hodge}
Let $M \in \Db \MHM(A)$ be a complex of mixed Hodge modules on a complex abelian
variety $A$. Then all cohomology support loci of $\rat M$ are stable under the action
of $\GalCQ$ on $\Char(A)$.  
\end{proposition}

\begin{proof}
The natural $\Lambda$-action on the group ring $\kLambda$ gives rise to a local
system of $k$-vector spaces $\shLk$ on the abelian variety. The discussion in
\cite[Section~14]{Schnell} shows that the cohomology support loci of $\rat M$ are
computed by the complex
\[
	\derR \pl \bigl( \rat M \tensor_{\QQ} \shLC \bigr) \in \Dbcoh \bigl( \CLambda \bigr),
\]
where $p \colon A \to \pt$ denotes the morphism to a point. In the case at hand, 
\[
	\derR \pl \bigl( \rat M \tensor_{\QQ} \shLC \bigr)
		\simeq \derR \pl \bigl( \rat M \tensor_{\QQ} \shLQ \bigr) 
		\tensor_{\QLambda} \CLambda
\]
is obtained by extension of scalars from a complex of $\QLambda$-modules
\cite[Proposition~14.7]{Schnell}; this means that all cohomology support loci of
$M$ are defined over $\QQ$, and therefore stable under the $\GalCQ$-action
on $\Char(A)$.
\end{proof}

\subsection{The Fourier-Mukai transform}

In this section, we review a few results about Fourier-Mukai transforms of holonomic
$\Dmod_A$-modules from \cite{Schnell}. The Fourier-Mukai transform, introduced by
Laumon \cite{Laumon} and Rothstein \cite{Rothstein}, is an equivalence of
categories
\[
	\FM_A \colon \Dbcoh(\Dmod_A) \to \Dbcoh(\OAsh);
\]
for a single coherent $\Dmod_A$-module $\Mmod$, it is defined by the formula
\[
	\FM_A(\Mmod) = \derR (p_2)_{\ast} \DR_{A \times \Ash/\Ash} \bigl(
		p_1^{\ast} \Mmod \tensor (\Psh, \nablash) \bigr),
\]
where $(\Psh, \nablash)$ is the universal line bundle with connection on $A \times \Ash$. 

The Fourier-Mukai transform satisfies several useful exchange formulas
\cite[Section~3.3]{Laumon}; recall that for $f \colon A \to B$ a homomorphism of abelian
varieties, 
\[
	\fp \colon \Dbcoh(\Dmod_A) \to \Dbcoh(\Dmod_B) \quad \text{and} \quad
		\fsi \colon \Dbcoh(\Dmod_B) \to \Dbcoh(\Dmod_A)
\] 
denote, respectively, the direct image and the shifted inverse image functor, while
$\DA \colon \Dbcoh(\Dmod_A) \to \Dbcoh(\Dmod_A)^{\opp}$ is the duality functor. 

\begin{theorem} \label{thm:Laumon} 
Let $\Mmod, \Mmod_1, \Mmod_2 \in \Dbcoh(\Dmod_A)$ and $\Nmod \in \Dbcoh(\Dmod_B)$.
\begin{enumerate}[(a)]
\renewcommand{\theenumi}{(\alph{enumi})}
\item For any homomorphism of abelian varieties $f \colon A \to B$, one has
\begin{align*}
	\derL (\fsh)^{\ast} \FM_A(\Mmod) &\simeq \FM_B \bigl( \fp \Mmod \bigr), \\
	\derR \fsh_{\ast} \FM_B(\Nmod) &\simeq \FM_A \bigl( \fsi \Nmod \bigr).
\end{align*}

\item One has $\FM_A \bigl( \DA \Mmod \bigr) \simeq \langle -1_{\Ash} \rangle^{\ast} \, 
		\derR \shHom \bigl( \FM_A(\Mmod), \OAsh \bigr)$.

\item Let $m \colon A \times A \to A$ be the addition morphism. Then one has
\[
	\FM_A \bigl( \mp(\Mmod_1 \boxtimes \Mmod_2) \bigr) \simeq 
		\FM_A(\Mmod_1) \Ltensor_{\OAsh} \FM_A(\Mmod_2).
\]
\end{enumerate}
\end{theorem}

Now let $\Dbh(\Dmod_A)$ be the full subcategory of $\Dbcoh(\Dmod_A)$ consisting of
cohomologically bounded and holonomic complexes. We already mentioned that the
cohomology support loci $S_m^k(A, \Mmod)$ of a holonomic complex are finite unions of
linear subvarieties; here is another result from \cite{Schnell} that will be used
below.

\begin{theorem} \label{thm:t-structure}
Let $\Mmod$ be a holonomic $\Dmod_A$-module. Then $\FM_A(\Mmod) \in \Dtcoh{\geq
0}(\OAsh)$, and for any $\ell \geq 0$, one has $\codim \Supp \shH^{\ell}
\FM_A(\Mmod) \geq 2 \ell$.
\end{theorem}

The precise relationship between the support of $\FM_A(\Mmod)$ and the cohomology
support loci of $\Mmod$ is given by the base change theorem, which implies that, for
every $n \in \ZZ$, one has
\begin{equation} \label{eq:base-change}
	\bigcup_{k \geq n} \Supp \shH^k \FM_A(\Mmod) = \bigcup_{k \geq n} S_1^k(A, \Mmod)
\end{equation}
In particular, the support of the Fourier-Mukai transform $\FM_A(\Mmod)$ is equal to
the union of all the cohomology support loci of $\Mmod$.

\section{Proof of the theorem}

Consider a complex of mixed Hodge modules $M \in \Db \MHM(A)$
that admits a $\ZZ$-structure, and denote by $\rat M \in \Dbc(\QQ_A)$ the underlying
complex of constructible sheaves. To prove \theoremref{thm:main}, we have to show
that that all cohomology support loci of $M$ are complete unions of arithmetic
subvarieties of $\Char(A)$.

\subsection{Reduction steps}

Our first task is to show that every $S_m^k(A, M)$ is a finite union of arithmetic
subvarieties. The proof is by induction on the dimension of $A$; we may therefore
assume that \emph{the theorem is valid on every abelian variety of strictly smaller
dimension.} This has several useful consequences.

\begin{lemma} \label{lem:reduction}
Let $f \colon A \to B$ be a homomorphism from $A$ to a lower-dimensional abelian
variety $B$. Then every intersection 
\[
	S_m^k(A, M) \cap \im \Char(f)
\] 
is a finite union of arithmetic subvarieties.
\end{lemma}

\begin{proof}
The complex $\fl M \in \Db \MHM(B)$ again admits a $\ZZ$-structure. If we now tensor
by points of $\Char(B)$ and take cohomology, we find that
\[
	\Char(f)^{-1} S_m^k(A, M) = S_m^k(B, \fl M).
\]
By induction, we know that the right-hand side is a finite union of arithmetic
subvarieties of $\Char(B)$; consequently, the same is true for the intersection
$S_m^k(A, M) \cap \im \Char(f)$.
\end{proof}

The inductive assumption lets us to show that all positive-dimensional components
of the cohomology support loci of $M$ are arithmetic.

\begin{lemma} \label{lem:positive}
Let $Z$ be an irreducible component of some $S_m^k(A, M)$. If $\dim Z \geq 1$,
then $Z$ is an arithmetic subvariety of $\Char(A)$.
\end{lemma}

\begin{proof}
Since $Z$ is a linear subvariety, it suffices to prove that $Z$ contains a torsion
point. Now $A$ is an abelian variety, and so we can find a surjective homomorphism $f
\colon A \to B$ to an abelian variety of dimension $\dim A - \dim Z/2$, such that $Z
\cap \im \Char(f)$ is a finite set of points. According to \lemmaref{lem:reduction},
the intersection is a finite union of arithmetic subvarieties, hence a finite set of
torsion points. In particular, $Z$ contains a torsion point, and is therefore an
arithmetic subvariety of $\Char(A)$.
\end{proof}

Irreducible components that are already contained in a proper arithmetic subvariety
of $\Char(A)$ can also be handled by induction.

\begin{lemma} \label{lem:contained}
Let $Z$ be an irreducible component of $S_m^k(A, M)$. If $Z$ is contained in a
proper arithmetic subvariety of $\Char(A)$, then $Z$ is itself an arithmetic subvariety.
\end{lemma}

\begin{proof}
It again suffices to show that $Z$ contains a torsion point. For some $n \geq 1$,
there is a torsion point of order $n$ on the arithmetic subvariety that contains $Z$.
After pushing forward by the multiplication-by-$n$ morphism $\langle n
\rangle \colon \Char(A) \to \Char(A)$, which corresponds to replacing $M$ by its
inverse image $\langle n \rangle^{\ast} M$ under $\langle n \rangle \colon A \to A$,
we can assume that $Z \subseteq \im \Char(f)$, where $f \colon A \to B$ is a morphism
to a lower-dimensional abelian variety. The assertion now follows from
\lemmaref{lem:reduction}.
\end{proof}

The following result allows us to avoid cohomology in degree $0$.

\begin{lemma} \label{lem:k}
Let $M \in \MHM(A)$, and let $Z$ be an irreducible component of some cohomology
support locus of $M$. If $Z \neq \Char(A)$, then $Z$ is contained in
$S_m^k(A, M)$ for some $k \neq 0$ and some $m \geq 1$.
\end{lemma}

\begin{proof}
This follows easily from the fact that the Euler characteristic
\[
	\chi \bigl( A, \rat M \tensor_{\QQ} \CCrho \bigr)
		= \sum_{k \in \ZZ} (-1)^k \dim H^k \bigl( A, 
			\rat M \tensor_{\QQ} \CCrho \bigr)
\]
is independent of the point $\rho \in \Char(A)$.
\end{proof}

\subsection{Torsion points on components}

Let $Z$ be an irreducible component of some cohomology support locus of $M$. If $\dim
Z \geq 1$, \lemmaref{lem:positive} shows that $Z$ is an arithmetic subvariety; we may
therefore assume that $Z = \{\rho\}$ consists of a single point. We have to prove
that $\rho$ has finite order in $\Char(A)$. There are three steps.

\paragraph{Step 1}
We begin by reducing the problem to the case where $M$ is a single mixed Hodge
module. Each of the individual cohomology modules $\shH^q(M) \in \MHM(A)$ also admits
a $\ZZ$-structure (by \lemmaref{lem:shH}); we know by induction that all
positive-dimensional irreducible
components of its cohomology support loci are arithmetic subvarieties. If $\rho$ is
contained in such a component, \lemmaref{lem:contained} proves that $\rho$ is a
torsion point; we may therefore assume that whenever there is some $p \neq 0$ such
that $H^p \bigl( A, \rat \shH^q(M) \tensor_{\QQ} \CCrho \bigr)$ is nontrivial, $\rho$
is an isolated point of the corresponding cohomology support locus.
To exploit this fact, let us consider the spectral sequence
\[
	E_2^{p,q} = H^p \bigl( A, \rat \shH^q(M) \tensor_{\QQ} \CCrho \bigr)
		\Longrightarrow H^{p+q} \bigl( A, \rat M \tensor_{\QQ} \CCrho \bigr).
\]
If $E_2^{p,q} \neq 0$ for some $p \neq 0$, then $\rho$ must be an isolated point in
some cohomology support locus of $\shH^q(M)$; in that case, we can replace $M$ by the
single mixed Hodge module $\shH^q(M)$. If $E_2^{p,q} = 0$ for every $p \neq 0$, then
the spectral sequence degenerates and
\[
	H^k \bigl( A, \rat M \tensor_{\QQ} \CCrho \bigr)
	\simeq H^0 \bigl( A, \rat \shH^k(M) \tensor_{\QQ} \CCrho \bigr).
\]
But $\rho \in S_m^k(A, M)$ is an isolated point, and so by semi-continuity, it must
also be an isolated point in $S_m^0 \bigl( A, \shH^k(M) \bigr)$; again, we can
replace $M$ by the single mixed Hodge module $\shH^k(M)$.

\paragraph{Step 2}

We now construct \emph{another} mixed Hodge module with $\ZZ$-structure, such that
the union of all cohomology support loci contains $\rho$ but is not equal to
$\Char(A)$. We can then use the inductive hypothesis to reduce the problem to the case where
$\CC_{\rho^{-1}} \decal{\dim A}$ is a direct factor of $\CC \tensor_{\QQ} \rat M$;
because of \lemmaref{lem:factor}, this will be sufficient to conclude that the
character $\rho$ has finite order.

The idea for the construction comes from a recent article by Kr\"amer and Weissauer
\cite[Section~13]{KW}. Since $M \in \MHM(A)$ is a single mixed Hodge module, we can
use \lemmaref{lem:k} to arrange that $\rho \in S_m^k(A, M)$ is an isolated point for
some $k \neq 0$ and some $m \geq 1$; to simplify the argument, we shall take the
absolute value $\abs{k}$ to be as large as possible. Now let $A \times \dotsm
\times A$ denote the $d$-fold product of $A$ with itself, and let $m \colon A \times
\dotsm \times A \to A$ be the addition morphism. The $d$-fold exterior product $M
\boxtimes \dotsm \boxtimes M$ is mixed Hodge module on $A \times \dotsm \times A$,
and clearly inherits a $\ZZ$-structure from $M$. Setting
\[
	M_d = m_{\ast} \bigl( M \boxtimes \dotsm \boxtimes M \bigr) \in \Db \MHM(A),
\]
it is easy to see from our choice of $k$ that
\[
	H^{kd} \bigl( A, \rat M_d \tensor_{\QQ} \CCrho \bigr) \simeq
	H^k \bigl( A, \rat M \tensor_{\QQ} \CCrho \bigr)^{\tensor d}.
\]
The right-hand side is nonzero, and so $\rho \in S_{md}^{kd}(A, M_d)$. By a similar
spectral sequence argument as above, we must have
\[
	H^p \bigl( A, \rat \shH^q(M_d) \tensor_{\QQ} \CCrho \bigr) \neq 0
\]
for some $p,q \in \ZZ$ with $p+q = kd$ and $-g \leq p \leq g$. If we take $d > g$,
this forces $q \neq 0$. In other words, we can find $q \neq 0$ such that $\rho$ lies
in some cohomology support locus of the mixed Hodge module $\shH^q(M_d)$.

\begin{lemma}
If $q \neq 0$, all nontrivial cohomology support loci of $\shH^q(M_d)$ are
properly contained in $\Char(A)$.  
\end{lemma}

\begin{proof}
It suffices to prove this for the underlying regular holonomic $\Dmod$-module
$\shH^q \mp(\Mmod \boxtimes \dotsm \boxtimes \Mmod)$. The properties of the
Fourier-Mukai transform in \theoremref{thm:Laumon} imply that
\[
	\FM_A \bigl( \mp(\Mmod \boxtimes \dotsm \boxtimes \Mmod) \bigr)
	\simeq \FM_A(\Mmod) \Ltensor_{\OAsh} \dotsm \Ltensor_{\OAsh} \FM_A(\Mmod),
\]
and all cohomology sheaves of this complex, except possibly in degree $0$, are
torsion sheaves (by \theoremref{thm:t-structure}). In the spectral sequence
\begin{align*}
	E_2^{p,q} = \shH^p \FM_A \bigl( &\shH^q 
		\mp(\Mmod \boxtimes \dotsm \boxtimes \Mmod) \bigr) \\
	&\Longrightarrow \shH^{p+q} \FM_A 
		\bigl( \mp(\Mmod \boxtimes \dotsm \boxtimes \Mmod) \bigr),
\end{align*}
the sheaf $E_2^{p,q}$ is zero when $p < 0$, and torsion when $p > 0$, for the same
reason. It follows that $E_2^{0,q}$ is also a torsion sheaf for $q \neq 0$, which
proves the assertion.
\end{proof}

\paragraph{Step 3}

Now we can easily finish the proof. The mixed Hodge module $\shH^q(M_d)$ again admits
a $\ZZ$-structure by \lemmaref{lem:shH}; by induction, all positive-dimensional
irreducible components of its cohomology support loci are proper arithmetic
subvarieties of $\Char(A)$. If $\rho$ is contained in one of them, we are done by
\lemmaref{lem:contained}. After replacing $M$ by $\shH^q(M_d)$, we can
therefore assume that, whenever $H^k \bigl( A, \rat M \tensor_{\QQ} \CCrho \bigr)$ is
nontrivial, $\rho$ is an isolated point of the corresponding cohomology support
locus. Note that we now have this for all values of $k \in \ZZ$, including $k = 0$.

Let $\Mmod$ denote the regular holonomic $\Dmod$-module underlying the mixed Hodge
module $M$.  If $(L, \nabla) \in \Ash$ is the flat line bundle corresponding to our character
$\rho$, the assumptions on $M$ guarantee that $(L, \nabla)$ is an isolated point in
the support of $\FM_A(\Mmod)$. This means that, in the derived category,
$\FM_A(\Mmod)$ has a direct factor supported on the point $(L, \nabla)$. But the
Fourier-Mukai transform is an equivalence of categories, and so $\Mmod \simeq \Mmod'
\oplus \Mmod''$, where $\Mmod'$ is a regular holonomic $\Dmod$-module whose
Fourier-Mukai transform is supported on $(L, \nabla)$. It is well-known that $\Mmod'$
is the tensor product of $(L, \nabla)^{-1}$ and a unipotent flat vector bundle; in
particular, $\Mmod$ contains a  sub-$\Dmod$-module isomorphic to $(L, \nabla)^{-1}$.
Equivalently, $\CC \tensor_{\QQ} \rat M$ has a subobject isomorphic to
$\CC_{\rho^{-1}} \decal{\dim A}$. Because the cohomology support loci of $M$ are
stable under the $\GalCQ$-action on $\Char(A)$ (by
\propositionref{prop:Galois-Hodge}), the same is true for every conjugate $g \rho$,
where $g \in \GalCQ$. We can now apply \lemmaref{lem:factor} to show that $\rho$ must
be a torsion point.  

\vspace{\baselineskip}

This concludes the proof that all cohomology support loci of $M$ are finite
unions of arithmetic subvarieties of $\Char(A)$.

\subsection{Completeness of the set of components}

We finish the proof of \theoremref{thm:main} by showing that each cohomology support
locus of $M$ is a \emph{complete} union of arithmetic subvarieties of $\Char(A)$.
The argument is based on the following simple criterion for completeness.

\begin{lemma} \label{lem:Galois-complete}
A finite union of arithmetic subvarieties of $\Char(A)$ is complete if and only if it
is stable under the action by $\GalCQ$.
\end{lemma}

\begin{proof}
For a point $\tau \in \Char(A)$ of order $n$, the orbit under the group $G = \GalCQ$
consists precisely of the characters $\tau^k$ with $\gcd(k,n) = 1$; consequently, a
complete collection of arithmetic subvarieties is stable under the $G$-action.
To prove the converse, let $Z$ be a finite union of arithmetic subvarieties stable
under the action by $G$. Let $\tau L$ be one of its components; here $L$ is a linear
subvariety and $\tau \in \Char(A)$ a point of order $n$, say. Let $p$ be any prime
number with $\gcd(n,p) = 1$, and denote by $\Lp$ the set of points of order $p$.
For any character $\rho \in \Lp$, we have $\ord(\tau \rho) = np$; the $G$-orbit of the
set $\tau \Lp$ is therefore equal to
\[
	(G \tau) \cdot \Lp = \menge{\tau^k}{\gcd(k,n) = 1} \cdot \Lp.
\]
Because the union of all the finite subsets $\Lp$ with $\gcd(n,p) = 1$ is dense in
the linear subvariety $L$, it follows that 
\[
	\menge{\tau^k}{\gcd(k,n) = 1} \cdot L \subseteq Z;
\]
this proves that $Z$ is complete.
\end{proof}

\begin{theorem}
Let $M \in \Db \MHM(A)$ be a complex of mixed Hodge modules that admits a
$\ZZ$-structure.  Then all cohomology support loci of $M$ are complete collections of
arithmetic subvarieties of $\Char(A)$.  
\end{theorem}

\begin{proof}
We already know that each $S_m^k(A, M)$ is a finite union of arithmetic subvarieties
of $\Char(A)$. By \propositionref{prop:Galois-Hodge}, it is stable under the
$\GalCQ$-action on $\Char(A)$; we can now apply \lemmaref{lem:Galois-complete} to
conclude that $S_m^k(A, M)$ is complete.
\end{proof}


  

\begin{thebibliography}{18}
\expandafter\ifx\csname natexlab\endcsname\relax\def\natexlab#1{#1}\fi
\expandafter\ifx\csname selectlanguage\endcsname\relax
  \def\selectlanguage#1{\relax}\fi

\bibitem[\protect\citename{Arapura, }1992]{Arapura}
Arapura, Donu. 1992.
\newblock Higgs line bundles, {G}reen-{L}azarsfeld sets, and maps of {K}\"ahler
  manifolds to curves.
\newblock {\em Bull. Amer. Math. Soc. (N.S.)}, {\bf 26}(2), 310--314.

\bibitem[\protect\citename{Beauville, }1992]{Beauville}
Beauville, Arnaud. 1992.
\newblock Annulation du {$H^1$} pour les fibr\'es en droites plats.
\newblock {Pages  1--15 of:} {\em Complex algebraic varieties ({B}ayreuth,
  1990)}.
\newblock Lecture Notes in Math., vol. 1507.
\newblock Berlin: Springer.

\bibitem[\protect\citename{Be{\u\i}linson {et~al.},
  }1982]{BBD}
Be{\u\i}linson, A.~A., Bernstein, J., and Deligne, P. 1982.
\newblock Faisceaux pervers.
\newblock {Pages  5--171 of:} {\em Analysis and topology on singular spaces,
  {I} ({L}uminy, 1981)}.
\newblock Ast\'erisque, vol. 100.
\newblock Paris: Soc. Math. France.

\bibitem[\protect\citename{Deligne, }1971]{Deligne-II}
Deligne, P. 1971.
\newblock Th\'eorie de {H}odge. {II}.
\newblock {\em Inst. Hautes \'Etudes Sci. Publ. Math.}, {\bf 40}, 5--57.

\bibitem[\protect\citename{Deligne, }1987]{Deligne}
Deligne, P. 1987.
\newblock Un th\'eor\`eme de finitude pour la monodromie.
\newblock {Pages  1--19 of:} {\em Discrete groups in geometry and analysis
  ({N}ew {H}aven, {C}onn., 1984)}.
\newblock Progr. Math., vol. 67.
\newblock Boston, MA: Birkh\"auser Boston.

\bibitem[\protect\citename{Green and Lazarsfeld, }1987]{GL1}
Green, Mark, and Lazarsfeld, Robert. 1987.
\newblock Deformation theory, generic vanishing theorems, and some conjectures
  of {E}nriques, {C}atanese and {B}eauville.
\newblock {\em Invent. Math.}, {\bf 90}(2), 389--407.

\bibitem[\protect\citename{Green and Lazarsfeld, }1991]{GL2}
Green, Mark, and Lazarsfeld, Robert. 1991.
\newblock Higher obstructions to deforming cohomology groups of line bundles.
\newblock {\em J. Amer. Math. Soc.}, {\bf 4}(1), 87--103.

\bibitem[\protect\citename{Hacon, }2004]{Hacon}
Hacon, Christopher~D. 2004.
\newblock A derived category approach to generic vanishing.
\newblock {\em J. Reine Angew. Math.}, {\bf 575}, 173--187.

\bibitem[\protect\citename{Juteau, }2009]{Juteau}
Juteau, Daniel. 2009.
\newblock Decomposition numbers for perverse sheaves.
\newblock {\em Ann. Inst. Fourier (Grenoble)}, {\bf 59}(3), 1177--1229.

\bibitem[\protect\citename{Kr{\"a}mer and Weissauer,
  }2011]{KW}
Kr{\"a}mer, Thomas, and Weissauer, Rainer. 2011.
\newblock Vanishing theorems for constructible sheaves on abelian varieties.
\newblock{\texttt{arXiv:1111.4947}}

\bibitem[\protect\citename{Laumon, }1996]{Laumon}
Laumon, G{\'e}rard. 1996.
\newblock Transformation de {F}ourier g{\'e}n{\'e}ralis{\'e}e.
\newblock{\texttt{arXiv:alg-geom/9603004}}

\bibitem[\protect\citename{Pareschi, }2012]{Pareschi}
Pareschi, Giuseppe. 2012.
\newblock Basic results on irregular varieties via {F}ourier-{M}ukai methods.
\newblock {Pages  379--403 of:} {\em Current developments in algebraic
  geometry}.
\newblock Math. Sci. Res. Inst. Publ., vol. 59.
\newblock Cambridge: Cambridge Univ. Press.

\bibitem[\protect\citename{Pink and Roessler, }2004]{PinkRoessler}
Pink, Richard, and Roessler, Damian. 2004.
\newblock A conjecture of {B}eauville and {C}atanese revisited.
\newblock {\em Math. Ann.}, {\bf 330}(2), 293--308.

\bibitem[\protect\citename{Popa and Schnell, }2013]{PopaSchnell}
Popa, Mihnea, and Schnell, Christian. 2013.
\newblock Generic vanishing theory via mixed {H}odge modules.
\newblock {\em Forum of Mathematics, Sigma}, {\bf 1}(e1), 1--60.

\bibitem[\protect\citename{Rothstein, }1996]{Rothstein}
Rothstein, Mitchell. 1996.
\newblock Sheaves with connection on abelian varieties.
\newblock {\em Duke Math. J.}, {\bf 84}(3), 565--598.

\bibitem[\protect\citename{Saito, }1988]{Saito-HM}
Saito, Morihiko. 1988.
\newblock Modules de {H}odge polarisables.
\newblock {\em Publ. Res. Inst. Math. Sci.}, {\bf 24}(6), 849--995.

\bibitem[\protect\citename{Saito, }1990]{Saito-MHM}
Saito, Morihiko. 1990.
\newblock Mixed {H}odge modules.
\newblock {\em Publ. Res. Inst. Math. Sci.}, {\bf 26}(2), 221--333.

\bibitem[\protect\citename{Schnell, }2014]{Schnell}
Schnell, Christian. 2014.
\newblock Holonomic $\mathcal{D}$-modules on abelian varieties.
\newblock {\em Inst. Hautes \'Etudes Sci. Publ. Math.}.
\newblock{\texttt{10.1007/s10240-014-0061-x}}

\bibitem[\protect\citename{Simpson, }1990]{Simpson-RH}
Simpson, Carlos~T. 1990.
\newblock Transcendental aspects of the {R}iemann-{H}ilbert correspondence.
\newblock {\em Illinois J. Math.}, {\bf 34}(2), 368--391.

\bibitem[\protect\citename{Simpson, }1993]{Simpson}
Simpson, Carlos~T. 1993.
\newblock Subspaces of moduli spaces of rank one local systems.
\newblock {\em Ann. Sci. \'Ecole Norm. Sup. (4)}, {\bf 26}(3), 361--401.

\bibitem[\protect\citename{Wang, }2013]{Wang}
Wang, Botong. 2013.
\newblock Torsion points on the cohomology jump loci of compact K\"ahler manifolds.
\newblock{\texttt{arXiv:1312.6619}}

\end{thebibliography}
  \end{document}